\documentclass[12pt]{amsart}
\usepackage{amsmath, amsthm, amscd, amsfonts, amssymb, graphicx, color}
\usepackage[bookmarksnumbered, colorlinks, plainpages]{hyperref}

\newtheorem{theorem}{Theorem}[section]
\newtheorem{thm}[theorem]{Theorem}
\newtheorem{lem}[theorem]{Lemma}
\newtheorem{prop}[theorem]{Proposition}
\newtheorem{cor}[theorem]{Corollary}
\newtheorem{rem}[theorem]{Remark}

\newcommand{\Ann}{\mbox{Ann}\,}

\newcommand{\Hom}{\mbox{Hom}\,}

\newcommand{\Spec}{\mbox{Spec}\,}

\newcommand{\Ass}{\mbox{Ass}}

\newcommand{\Assh}{\mbox{Assh}\,}
\newcommand{\Att}{\mbox{Att}\,}
\newcommand{\Supp}{\mbox{Supp}\,}

\renewcommand{\dim}{\mbox{dim}\,}

\newcommand{\cd}{\mbox{cd}\,}

\newcommand{\Min}{\mbox{Min}\,}

\newcommand{\N}{\mathbb{N}}

\newcommand{\fa}{\mathfrak{a}}

\newcommand{\fm}{\mathfrak{m}}
\newcommand{\fr}{\mathfrak{r}}
\newcommand{\fp}{\mathfrak{p}}
\newcommand{\fq}{\mathfrak{q}}

\numberwithin{equation}{section}

\begin{document}

\title{On the  Attached prime ideals of local
cohomology modules defined by a pair of ideals}

\author[Z. Habibi]{Zohreh Habibi}
\address{Zohreh Habibi \\Payame Noor University, Po Box 19395-3697, Tehran, Iran}
\email{z\_habibi@pnu.ac.ir}

\author[M. Jahangiri]{Maryam Jahangiri }
\address{ Maryam Jahangiri\\Faculty of Mathematical Sciences and Computer,  Kharazmi
University,  Tehran, Iran AND  Institute for Research in Fundamental Sciences (IPM)  P. O.  Box: 19395-5746, Tehran, Iran.} \email{jahangiri@khu.ac.ir}

\author[Kh. Ahmadi Amoli]{Khadijeh Ahmadi Amoli}
\address{ Khadijeh Ahmadi Amoli\\Payame Noor University, Po Box 19395-3697, Tehran, Iran}

\email{khahmadi@pnu.ac.ir} 

\thanks{The second author was in part supported by a grant from IPM (No.
92130111)}

\maketitle

\maketitle

\begin{center}
\end{center}

\begin{abstract}  Let $I$ and $J$ be two ideals of a commutative Noetherian ring $R$ and $M$
be an $R$-module of dimension $d$. If $R$ is a complete local ring and
$M$ is  finite, then attached prime ideals of $H^{d-1}_{I,J}(M)$
are computed by means of the concept of co-localization. Also, we
illustrate the attached prime  ideals of $H^{t}_{I,J}(M)$ on a
non-local ring $R$, for $t= \dim M$ and $t=  \cd(I,J,M)$.
\end{abstract}

\vskip 0.2 true cm


\pagestyle{myheadings}
\markboth{\rightline {\scriptsize  Habibi, Jahangiri and Ahmadi
Amoli}}
         {\leftline{\scriptsize On the set of Attached prime...}}

\bigskip
\bigskip


\section{\bf Introduction}

Throughout this paper, $R$ denotes a commutative
Noetherian ring, $M$ an $R$-module and $I$ and $J$ stand for two
ideals of $R$.  For all $i\in \N_0$ the $i$-th
local cohomology functor with respect to $(I,J)$, denoted by
$H^{i}_{I,J}(-)$, defined by Takahashi et. all in \cite{TAK} as the
$i$-th right derived functor of the $(I,J)$- torsion functor $\Gamma
_{I,J}(-)$, where $$\Gamma _{I,J}(M):=\{x \in M : I^{n}x\subseteq Jx
\  \text {for} \   n\gg 1\}.$$ This notion coincides  with  the
ordinary local cohomology functor $H^{i}_{I }(-)$ when $J=0$, see
\cite{B-SH}.

The main motivation for this generalization comes from the study of
a dual of ordinary local cohomology modules $H^{i}_{I }(M)$
(\cite{sch}). Basic facts and more information about local
cohomology defined by a pair of ideals can be obtained from
\cite{TAK}, \cite{CH} and \cite{CH-W}.

The second section of this paper is devoted to study the attached prime ideals of local cohomology
modules with respect to a pair of ideals by means of
co-localization. The concept of co-localization introduced by
Richardson in \cite{RICH}.

Let $(R,\fm)$ be local and $M$ be a
finite $R$-module of dimension $d$. If $c$ is a non-negative integer
such that $H^{i}_{I,J}(R) = 0$ for all $i> c$ and $H^{c}_{I,J} (R)$
is representable, then we illustrate the attached prime ideals of
$^{\fp}H^{c}_{I,J}(M)$ (see Theorem \ref{th mainI}). In addition if
$R$ is complete, then we have made use of Theorem \ref{th mainI} to
prove that in a special case
$$\Att(H^{d-1}_{I,J}(M)) \subseteq T \cup \Assh(M)\ \  \text{and} \  \  T\subseteq
\Att(H^{d-1}_{I,J}(M)),$$ where $$T=\{\fp \in \Supp(M) : \dim M/\fp M
= d -1, J\subseteq \fp \  and \ \sqrt{I + \fp }= \fm\},$$ (see Theorem
\ref{th mainII}).

In \cite[Theorem 2.1]{CH} the set of  attached
prime ideals of $H^{dim M}_{I,J}(M)$ was computed on a local ring. We generalize
this theorem to the non-local case. Also, the authors in
\cite[2.4]{D-YII} specified a subset of attached prime ideals of
ordinary top local cohomology module  $H^{cd(I, M)}_{I }(M)$. We improve it
for $H^{cd(I,J,M)}_{I,J}(M)$ over a not necessarily local ring,
where $\cd(I,J,M)= \sup \{i\in \mathbb{N}_{0}: H^{i}_{I,J}(M)\neq 0
\}$ with the convention that $\cd(I,M)=\cd(I,0,M)$.


\section{\bf Attached prime ideals}
\vskip 0.4 true cm

In this section we study the set of attached prime ideals of
local cohomology modules with respect to a pair of ideals.
\begin{rem}\label{ric}

\emph{Following \cite{RICH},  for a multiplicatively
closed subset $S$ of the local ring $(R,\fm)$, the co-localization
of $M$ relative to $S$ is defined to be the $S^{-1}R$-module
$S_{-1}(M):=D_{S^{-1}R}(S^{-1}D_{R}(M))$, where $D_R(-)$ is the Matlis
dual functor $\Hom_R(- , E_R(R/\fm))$. If $S = R \setminus \fp $ for some} \emph{$\fp \in
\Spec(R)$}\emph{, we write $^{\fp}M$ for $S_{-1}(M)$.}

\emph{Richardson in \cite[2.2]{RICH} proved that if $M$ is a
representable $R$- module, then so is $S_{-1}(M)$ and}
\emph{$\Att(S_{-1}M) = \{S^{-1}\fp : \fp \in
\Att(M)\}.$} \emph{Therefore, in order to get some results about
attached prime ideals of a module, it is convenient to study the
attached prime ideals of the co-localization of it.}\end{rem}


 \begin{lem}\label{lem befor main}
Let $(R,\fm)$ be a local ring, $\fa$ be an ideal of $R$ and $\fp\in
\Spec(R)$ with $\fa\subseteq \fp$. Let $R'=R/\fa$ and
$\fp'=\fp/\fa$. Then for any $R'$-module $X$ and $R'_{\fp'}$-module $Y$,
the following isomorphisms hold:

 $(i)$ \emph{$D_{R}(X)\cong D_{R'}(X) $} as
$R$-modules.

$(ii)$ \emph{$D_{R}(X)_{\fp}\cong D_{R'}(X)_{\fp'}$} as
$R_{\fp}$-modules.

$(iii)$ \emph{$D_{R_{\fp}}(Y) \cong D_{R'_{\fp'}}(Y)$} as
$R_{\fp}$-modules.

\end{lem}











In \cite[2.1 and 2.2]{EGH} the following theorems have been proved
for the attached prime ideals of $H^{d}_{I}(R)$ and $H^{d-1}_{I}(R)$ where $d=\dim R$.
 Here, we
generalize these theorems for the  local cohomology modules of $M$ with respect to a pair
of ideals when $M$ is a finite $R$-module  with $\dim
M=d$.

\begin{thm}\label{th mainI}
Let $(R,\fm)$ be a local ring, $M$ be a finite $R$-module, and
\emph{$\fp \in \Spec (R)$}. Assume that $c=cd(I,J,R)$ and
$H^{c}_{I,J} (R)$ is representable. Then

\begin{enumerate}
\item $\Att_{R_\fp}(^{\fp}H^{c}_{I,J}(M)) \subseteq \{\fq R_\fp : \dim M/\fq M
\geq c$, $\fq \subseteq \fp$, and $\fq \in \Spec (R)\}$.
\item If $R$ is complete, then
\begin{eqnarray*}
\Att_{R_\fp}(^{\fp}H^{dim M}_{I,J}(M)) =& \{&\fq R_\fp : \fq\in \Supp(M),\dim M/\fq M = \dim M ,
J\subseteq \fq\subseteq \fp,\\ & &and \sqrt{I + \fq} = \fm\}.
\end{eqnarray*}
\end{enumerate}
\end{thm}

\begin{proof}

 $(1)$ Let $\fq R_\fp\in \Att_{R_\fp} (^{\fp}H^{c}_{I,J}(M))$. By
\cite[3.1]{T-T-Y} and Remark $\ref{ric}$, we have
 $H^{c}_{I,J}(M)$ is representable and $\Att_{R_\fp} (^{\fp}H^{c}_{I,J}(M)) = \{\fq R_\fp : \fq\in
\Att(H^{c}_{I,J}(M)) \ and \ \fq\subseteq \fp\}$. Also, using
\cite[6.1.8]{B-SH} and \cite[2.11]{AGH-MEL}

 $$\begin{array}{ll}\Att(H^{c}_{I,J}(M/\fq M))&= \Att(H^{c}_{I,J}(M))\cap
\Supp(R/\fq).\end{array}$$ This implies that $H^{c}_{I,J}(M/\fq
M)\neq 0$ and consequently $\dim M/\fq M\geq c$.

 $(2)$ Let $\fp\in \Supp(M)$. Put $d:=\dim M$, $\overline{R} = R/\Ann_{R}M$, and $$T:= \{\fq R_\fp : \fq\in \Supp(M),\dim M/\fq M=d, J\subseteq \fq\subseteq \fp \  and \  \sqrt{I + \fq} = \fm\}.$$

Since $\dim_{\overline{R}}M=\dim_{R}M$, \cite[2.7]{TAK} and Lemma
$\ref{lem befor main}$ imply that
$^{\overline{\fp}}H^{d}_{I\overline{R},J\overline{R}}(M)\cong$ $
^{\fp}H^{d}_{I,J}(M)$, as $R_{\fp}$-modules. Therefore, by
\cite[8.2.5]{B-SH}, $\fq\in
 \Att_{\overline{R}_{\overline{\fp}}}(^{\overline{\fp}}H^{d}_{I\overline{R},J\overline{R}}(M))$
 if and only if $$\fq\cap R_{\fp}\in \Att_{R_{\fp}}(^{\overline{\fp}}H^{d}_{I\overline{R},J\overline{R}}(M))
= \Att_{R_\fp}(^{\fp}H^{d}_{I,J}(M)).$$

Now, without loss of generality, we may assume that $M$ is faithful
and $\dim R = d$.
 If $H^{d}_{I,J}(M)=0$, then
$\Att_{R_{\fp}}(^{\fp}H^{d}_{I,J}(M))=\emptyset$. Assume that $T\neq
\emptyset$ and $\fq R_{\fp}\in T$. Since $\dim M/\fq M = \dim R$, we
have $\dim R/\fq= d$. On the other hand, $\fq\in \Supp (M/JM)$.
Thus, by \cite[Theorem 2.4]{CH}, $\dim R/(I + \fq)> 0$ which
contradicts $\sqrt{I + \fq} = \fm$. So $T=\emptyset$.

Now, we assume that  $H^{d}_{I,J}(M)\neq0$.

  $\supseteq$: Let $\fq R_{\fp}\in T$. Since $H^{d}_{I,J}(M)$ is an Artinian $R$-module (cf. \cite[2.1]{CH-W}) so,  by Remark $\ref{ric}$, it is enough to show that $\fq\in \Att(H^{d}_{I,J}(M))$.
As $M/\fq M$ is $J$-torsion with dimension $d$ and $\sqrt{I + \fq}=
\fm$ , so by \cite[4.2.1 and 6.1.4]{B-SH}.
$$H^{d}_{I,J}(M/\fq M)\cong H^{d}_{I}(M/\fq M)\cong
H^{d}_{I(R/\fq)}(M/\fq M)\cong H^{d}_{\fm/\fq}(M/\fq M)\neq 0.$$
Hence \cite[6.1.8]{B-SH} and \cite[2.11]{AGH-MEL} imply that
$\emptyset\neq \Att(H^{d}_{I,J}(M/\fq M))= \Att(H^{d}_{I,J}(M)) \cap
\Supp(R/\fq).$
 Let $\fq_{0}\in \Att(H^{d}_{I,J}(M))$ be
such that $\fq\subset \fq _{0}$. So that $\dim M/\fq_{0}M < d$. On
the other hand, by Remark $\ref{ric}$, $\fq_{0}R_{\fq_0} \in
\Att_{R_{\fq_0}} (^{\fq_0}H^{d}_{I,J}(M))$ and this implies that
$\dim M/\fq_{0}M \geq d$ which is a contradiction. So $\fq= \fq
_{0}$.

$\subseteq$: Let $\fq R_{\fp}\in \Att_{R_\fp}
(^{\fp}H^{d}_{I,J}(M))$. As we have seen in the proof of part $(1)$,
$\dim M/\fq M = d$ and $\fq\subseteq \fp$. So by \cite[2.7]{TAK},
$$H^{d}_{IR/\fq,JR/\fq}(M/\fq M)\cong H^{d}_{I,J}(M/\fq M)\neq 0.$$

Now, by  \cite[Theorem 2.4]{CH}, there exists $\fr/\fq \in \Supp(R/\fq\otimes_{R/\fq}
\frac{M/\fq M}{(JR/\fq)(M/\fq M)})$ such that $\dim
\frac{R/\fq}{\fr/\fq}=d$ and $\dim \frac{R/\fq}{IR/\fq+\fr/\fq}=0 $.
Since $\fq R_{\fp} \in \Att_{R_{\fp}} (^{\fp}H^{d}_{I,J}(M))$, we
have $\fq \in \Att(H^{d}_{I,J}(M))$ and so $\fq\in \Supp(M)\cap
V(J)$. Hence $\fq/\fq\in \Supp_{R/\fq} (M/\fq M)$ and then
$$\dim R/\fq= \dim M/\fq M= d= \dim \frac{R/\fq}{\fr/\fq}=\dim R/\fq
.$$

Therefore, $\dim R/\fq = \dim R/\fr $ which shows that $\fq=\fr$.
Thus $\sqrt{I+\fq}=\fm$.
\end{proof}

\begin{rem}
\emph{The inclusion in Theorem $\ref{th mainI}$(1)  is not an equality
in general. Let the assumption be as in Theorem $\ref{th mainI}$.
Assume that $H^{d}_{I,J}(M)=0$, $\fp\in \Min(M)$ and $\dim M/\fp
M=d$. Then }\emph{$\Att_{R_\fp}
(^{\fp}H^{d}_{I,J}(M))=\emptyset$}\emph{. But}
$$\emph{$\{\fq R_\fp : \dim M/\fq M = d,\fq \subseteq \fp$ and
$\fq\in \Supp (M)\} = \{\fp R_{\fp}\}$.}$$

\end{rem}
\begin{thm}\label{th mainII}
Let $(R,\fm)$ be a complete local ring and $M$ be a finite
$R$-module with dimension $d$. Assume that $H^{i}_{I,J}(R) = 0$ for
all $i
> d-1$ and $H^{d-1}_{I,J} (R)$ is representable. Then
\begin{enumerate}
\item
\begin{eqnarray*}
\Att_R (H^{d-1}_{I,J}(M)) \subseteq &\{& \fp \in \Supp(M) :
\dim M/\fp M = d -1, J\subseteq \fp  \  and  \  \sqrt{I+\fp} =
\fm\} \\  & & \cup \Assh(M).
\end{eqnarray*}
\item
$$
\{\fp \in \Supp(M) : \dim M/\fp M = d
-1, J\subseteq \fp \  and  \  \sqrt{I + \fp }= \fm\}\subseteq \Att
(H^{d-1}_{I,J}(M)).$$
\end{enumerate}
\end{thm}

\begin{proof}
$(1)$ First we note that, by \cite[4.8]{TAK} and \cite[3.1]{T-T-Y},
$H^{d-1}_{I,J}(M)$ is representable and $\Att
(H^{d-1}_{I,J}(M))\subseteq \Supp (M)$. Now, let $\fp \in
\Att(H^{d-1}_{I,J}(M))$. Since $\fp R_\fp \in \Att_{R_\fp}
(^{\fp}H^{d-1}_{I,J}(M))$, by Theorem $\ref{th mainI}$ $(1)$, $\dim
M/\fp M \geq d-1$.

If $\dim M/\fp M=d$, then $\dim R/\fp=d$ and so $\fp \in \Assh(M)$.

Now, assume that $\dim M/\fp M = d-1$. Since $\fp \in
\Att(H^{d-1}_{I,J}(M))$, $H^{d-1}_{IR/\fp,JR/\fp}(M/\fp M)\cong
H^{d-1}_{I,J}(M/\fp M)\neq 0$. Thus, by \cite[Theorem 2.4]{CH}, there exists $\fr/\fp
\in \Supp(\frac{M/\fp M}{(JR/\fp)(M/\fp M)})$ such that $\dim
\frac{R}{\fr}=d$ and
 $\dim \frac{R}{I+\fr}=0 $. Hence $\fr=\fp , J\subseteq \fp,$ and $\sqrt{I+\fp}=\fm$.

 $(2)$ Let $\fp\in \Supp(M)$, $J\subseteq \fp$, $\dim M/\fp M = d-1$, and
$\sqrt{I + \fp}=\fm$. Then, by \cite[3.1]{T-T-Y} and Theorem
$\ref{th mainI}$ $(2)$,  $H^{d-1}_{I,J}(M)$ is representable, $\fp
R_\fp \in \Att_{R_\fp} (^{\fp}H^{d-1}_{I,J}(M/\fp M))$, and so
$\fp\in \Att (H^{d-1}_{I,J}(M/\fp M))$. Now, the proof is complete by considering the epimorphism
\\$H^{d-1}_{I,J}(M)\rightarrow H^{d-1}_{I,J}(M/\fp M)$.

\end{proof}
In the rest of the paper, following \cite{TAK},  we use the notations
$$W(I,J):= \{\fp\in Spec(R): I^{n}\subseteq \fp + J\
 for \  an \  integer \ n\geq1\}$$ and
$$\widetilde{W}(I,J):= \{\fa: \fa \  is \  an \  ideal \  of \ R; I^{n}\subseteq \fa+J \ for \ an \  integer \ n\geq
1\}.$$
 The following lemma can be proved using \cite[3.2]{TAK}.
\begin{lem}\label{supp}
For any non-negative integer $i$ and $R$-module $M$,

$(i)$ \emph{$\Supp(H^{i}_{I,J}(M))\subseteq \underset{\fa \in
\widetilde{W}(I,J)}{\bigcup}\Supp(H^{i}_{\fa}(M))$}.

$(ii)$ \emph{$\Supp(H^{i}_{I,J}(M))\subseteq \Supp(M) \cap W(I,J)$}.

\end{lem}




\begin{cor}\label{att}
Let $M$ be an $R$-module and $c=cd(I,J,R)$. Assume that $M$ is
representable or $H^{c}_{I,J}(R)$ is finite. Then
\emph{$$\Att(H^{c}_{I,J}(M))\subseteq \Att(M)\cap W(I,J).$$}
\end{cor}
\begin{proof}

By \cite[4.8]{TAK}, \cite[2.11]{AGH-MEL},\cite[3.1]{T-T-Y} and Lemma
\ref{supp} $(ii)$, we have
 $$\begin{array}{lll}\Att(H^{c}_{I,J}(M))&= \Att(M\otimes
H^{c}_{I,J}(R))&\subseteq \Att(M)\cap
\Supp(H^{c}_{I,J}(R))\\&&\subseteq \Att(M)\cap W(I,J).\end{array}$$
\end{proof}





Applying the set of attached prime ideals of top local cohomology module in
\cite[Theorem 2.2]{CH}, we obtain another presentation for it.
\begin{prop}\label{hat}

Let $(R,\fm)$ be a local ring and $\hat{R}$ denotes the $\fm-$adic
completion of $R$. Suppose that $M$ is a finite $R$-module of
dimension $d$. Then
\begin{eqnarray*}
\Att_{R}(H^{d}_{I,J} (M))=&\{&\fq\cap R : \fq \in
\Supp_{\hat{R}} (\hat{R}\otimes_{R}M/JM), \dim(\hat{R}/\fq) = d,\\
&&and \
  \dim\hat{R}/(I\hat{R} + \fq) = 0\}.
\end{eqnarray*}
\end{prop}
\begin{proof}
Denote the set of right hand side of the assertion by $T$. It is
clear that by\cite[Theorem 2.4]{CH}, $H^{d}_{I,J} (M)= 0$ if and only if
$T=\emptyset$. Assume that $H^{d}_{I,J} (M)\neq 0$ and $\fp\in \Supp
(M/JM)$ with the property that $\cd (I,R/\fp) = d$. Let $\fq\in
\Ass(M/JM)$ be such that $\fq\subseteq \fp$. Then $$d=\cd(I,R/\fp)
\leq \cd(I,R/\fq)\leq \dim R/\fq\leq \dim M/JM\leq \dim M= d $$

implies that $\fp=\fq\in \Ass(M/JM)$ and $\dim M/JM=d$. Now the
claim follows from \cite[3.10]{T-T-Y} and \cite[Theorem 2.1]{CH}.

\end{proof}
The following lemma, which can be proved by using the similar
argument of \cite[4.3]{TAK}, will be applied in the rest of the
paper.
\begin{lem}\label{4.3tak}
Let $M$ be a finite $R$-module. Suppose that $J\subseteq J(R)$,
where $J(R)$ denotes the Jacobson radical of $R$, and \emph{$\dim
M/JM=d$} be an integer. Then $H^{i}_{I,J}(M)=0$ for all $i>d$.

\end{lem}

Using Lemma \ref{4.3tak}, we can compute $\Att(H^{dim M}_{I,J}(M))$
in non-local case as a generalization of \cite[2.5]{DIV}.
\begin{prop}\label{j(r)}
Let $M$ be a finite $R$-module of dimension $d$ and $J\subseteq
J(R)$. Then
 \emph{$$ \begin{array}{ll}\Att(H^{d}_{I,J}(M))&=\Att(H^{d}_{I}(M/JM))\\ &=\{\fp\in
\Ass(M)\cap V(J): \cd(I,R/\fp) = d\}.\end{array}$$}
\end{prop}
\begin{proof}
The assertion holds by applying Lemma $\ref{4.3tak}$ and using the
same method of the proof of \cite[Theorem 2.1 and Proposition
2.1]{CH}.
\end{proof}

\begin{cor}\label{quotient}
Suppose that $J\subseteq J(R)$ and $M$ is a finite $R$-module such
that $\dim M=d$. Then \emph{$$\Att(\frac{H^{d}_{I,J}(M)}{J
H^{d}_{I,J}(M)}) =\{\fp\in \Supp (M)\cap V(J):\cd (I,R/\fp) =
d\}.$$}
\end{cor}
\begin{proof}
Let $\overline{R} = R/\Ann_{R}M$. Using \cite[2.7]{TAK},
$H^{d}_{I,J}(M)\cong H^{d}_{I\overline{R},J\overline{R}}(M)$ and
also for a prime $\fp\in \Supp (M)\cap V(J)$,
$\cd(I\overline{R},\overline{R}/\fp)=\cd(I,R/\fp)$. Thus we may
assume that $M$ is faithful and so $\dim R=d$. In virtue of
\cite[6.1.8]{B-SH}, $H^{d}_{I}(M/JM)\cong H^{d}_{I,J}(M/JM) \cong
\frac{H^{d}_{I,J}(M)}{J H^{d}_{I,J}(M)}$. Now, the assertion follows
by Proposition $\ref{j(r)}$.
\end{proof}

The final result of this section is a generalization of
\cite[2.4]{D-YII} in non-local case for local cohomology modules
with respect to a pair of ideals.
\begin{prop}\label{M-purity}
 Let  $J\subseteq J(R)$ and $M$ be a finite $R$-module. Then
 $$\{\fp\in \Ass(M)\cap V(J) : \cd(I,R/\fp)=\dim R/\fp=\cd(I,J,M)\} \subseteq \Att(H^{\cd(I,J,M)}_{I,J}(M)).$$
  Equality holds if
$\cd(I,J,M) = \dim M$.
\end{prop}

\begin{proof}
The same proof of \cite[2.4]{D-YII} remains valid by using
Proposition $\ref{j(r)}$.
\end{proof}



\bigskip
\bigskip

\end{document}